\documentclass{article}

\usepackage{amsmath,amsthm}
\usepackage{mwe}
\usepackage{enumitem}
\usepackage{amsfonts}
\usepackage{tikz}
\usepackage{mathrsfs}
\usepackage{graphicx}
\usepackage{picinpar}
\usepackage{amssymb}
\usepackage{mathtools}
\usepackage{csquotes}
\usepackage{upgreek}
\usepackage{graphicx, array, blindtext}
\usepackage[font=footnotesize,labelfont=bf,justification=centering]{caption}
\usepackage{wrapfig}
\usepackage{imakeidx}
\usepackage{titletoc}
\usepackage[nottoc,numbib]{tocbibind}
\usepackage[hyperindex]{hyperref}
\usepackage{subfiles}
\usepackage{scalerel}
\usepackage{mathtools}
\usepackage{graphicx, array, blindtext}
\usepackage{subcaption}
\usepackage{wrapfig}
\usepackage[numbers]{natbib}

\newtheoremstyle{dotless}{}{}{\itshape}{}{\bfseries}{}{ }{}
\theoremstyle{dotless}

\hypersetup{colorlinks,linkcolor=black,urlcolor=blue, citecolor=black,hyperindex=true}

\newcommand \eps {\varepsilon}

\newcommand \CC {\mathbb C}
\newcommand \R {\mathbb R}
\newcommand \RR {\mathbb R}

\newcommand \NN {\mathbb N}

\newcommand \ZZ {\mathbb Z}

\newcommand \QQ{\mathbb Q}

\newcommand \DD {\mathbb D}

\newcommand \cM {\mathcal{M}}

\newcommand \defin[1]{\emph{#1}}

\newcommand* \bb [1]{\left({#1}\right)}
\newcommand* \binr[2]{\left<{#1},{#2}\right>}
\newcommand* \bset[1] {\left\{{#1}\right\}}

\newcommand* \integrate[4]{\int_{#1}^{#2}{#3}d{#4}} 

\newcommand* \abs[1] {\left|{#1}\right|}

\newtheorem{thm}{Theorem}[section]
\newtheorem{lem}[thm]{Lemma}

\newtheorem{prop}[thm]{Proposition}
\newtheorem{cor}[thm]{Corollary}
\newtheorem{defn}[thm]{Definition}

\newtheorem{rmk}[thm]{Remark}

\newtheorem*{thm*}{Theorem}
\newtheorem*{lem*}{Lemma}
\newtheorem*{claim*}{Claim}
\newtheorem*{prop*}{Proposition}
\newtheorem*{cor*}{Corollary}
\newtheorem*{defn*}{Definition}
\newtheorem*{examp*}{Example}
\newtheorem*{rmk*}{Remark}
\newtheorem*{remind*}{Reminder}
\newtheorem*{exe*}{Exercise}
\newtheorem*{con*}{Conjecture}

\newtheorem{mainthm}{Theorem}

\renewcommand{\Re}{\operatorname{Re}}
\renewcommand{\Im}{\operatorname{Im}}

\usepackage[symbol]{footmisc}
\newcommand{\dist}{\operatorname{dist}}
\newcommand{\ignore}[1]{ }

\begin{document}
\title{Computability in Harmonic Analysis}
\author{Ilia Binder, Adi Glucksam, Cristobal Rojas, Michael Yampolsky}
\maketitle

\begin{abstract}
We study the question of constructive approximation of the harmonic measure $\omega_x^\Omega$ of a connected bounded domain $\Omega$ with respect to a point $x\in\Omega$. In particular, using a new notion of {\it computable harmonic approximation}, we show that for an arbitrary such $\Omega$, computability of the harmonic measure $\omega^\Omega_x$ for a single point $x\in\Omega$  implies computability of $\omega_y^\Omega$ for {\it any } $y\in \Omega$. This may require a different algorithm for different points $y$, which leads us to the construction of surprising natural examples of continuous functions that arise as solutions to a Dirichlet problem, whose values can be computed at any point, but cannot be computed with the use of the same  algorithm on all of their domain. We further study the conditions under which the harmonic measure is computable uniformly, that is by a single algorithm, and characterize them for regular domains with computable boundaries.
\end{abstract}
\section{Introduction}

\subsection{Motivation}
Questions of constructive approximation lie at the heart of modern  Analysis and its applications.   A classical early example is the first constructive proof of the Riemann Mapping Theorem given by P.~Koebe \cite{Koebe} in the early 1900's.
At the same time,  C.~Carath{\'e}odory \cite{Car12} published his celebrated theory of prime end boundary extension of the Riemann Mapping. A constructive approach to Carath{\'e}odory theory, developed recently by a subset of the authors in \cite{BRY-car},   leads to a delicate interplay of Complex Analysis and Computability theory. A natural object to consider next is the collection of the harmonic measures $$\{\omega_x^\Omega\}_{x\in\Omega}$$
of a domain $\Omega$, which, in the  case of a Jordan domain $\Omega$, are the pushforwards of the Lebesgue measure by the Carath{\'e}odory extensions of the Riemann Mappings (normalized at $x$)  from the unit circle to $\partial \Omega$.

In this paper, we consider the measures $\{\omega_x^\Omega\}$ from the constructive point of view for an arbitrary bounded connected domain $\Omega$ in an arbitrary dimension.
We want to understand what data about $\Omega$ is required in order to algorithmically approximate its harmonic measure in the weak sense, and, conversely, what we can learn about $\Omega$ from the ability to computably sample $\omega_x^\Omega$.

We encounter several surprises. Firstly, as our Theorem \ref{thm:allcompute} below demonstrates, the existence of an algorithm able to sample $\omega_x^\Omega$ at a single point $x = x_0\in\Omega$ is sufficient to guarantee the existence of such an algorithm for any other point $x$ in the domain $\Omega$. The proof of this theorem is based on a new concept of harmonic approximation of domains, which builds on the ideas of another work of the subset of the authors \cite{Binder2019}. We show that having an algorithm for approximating a domain harmonically is equivalent to being able to compute all of its harmonic measures.

Intriguingly, different algorithms may be required for different points $x\in\Omega$ in the above theorem, and we present an example of a domain $\Omega$ for which this is indeed the case in Theorem \ref{thm:nocompute_measure}. This naturally leads to an example of a Dirichlet problem for $\Omega$ whose unique solution $u(x)$ is a continuous real valued function which is piece-wise computable but not computable -- that is, different algorithms are required to approximate its values at different points $x$ (Theorem \ref{thm:nocompute_Dirichlet}). This is the first example we know of a ``natural'' such function and is certainly the first one known in constructive analysis.

This phenomenon does not occur under some natural assumptions on geometric computability of domain itself. In Theorem \ref{thm:uni_comp}, we essentially show that for such domains the existence of an algorithm sampling harmonic measure at one point implies the existence of an algorithm which computes harmonic measure at every point of the domain.

Moreover, in Theorem \ref{thm:computable_boundary} we find a necessary and sufficient condition for a regular domain with a computable boundary that guarantees the computability of its harmonic measure, the \emph{computable regularity}, defined below. The class of computably regular domains include simply connected planar domains and planar uniformly perfect domains.

As shown in \cite{BBRY}, harmonic measures can be sometimes uniformly computed for some natural domains with non-computable boundaries. In Theorem \ref{thm:lower_computable}, we produce a completely geometric characterization of all computably regular domains for which the harmonic measures can be computed.

It was proven in \cite{Binder2019} that for simply connected planar domains the weak convergence of harmonic measures is equivalent to the classical Carath{\'e}odory convergence. In light of these results, it is not surprising that the condition of Theorem \ref{thm:lower_computable} is exactly the same as the condition for the existence of computable Riemann bijection in \emph{Computable Riemann Mapping Theorem} (see \cite{BB}).

\subsection{Setting the stage}\label{stage} Let us briefly review the relevant notions of Computability Theory; for a detailed discussion, the reader is referred to \cite{BBRY} or \cite{BY-book}. As is standard, we formalize the notion of an algorithm, or a computer program, like a Turing Machine (TM), and we will use these terms interchangeably.
\begin{itemize}
\item A  function $f$ from $S\subset \QQ^d$ to $\QQ^m$ is \emph{computable} if there exists a TM  $\cM$ which, upon input $q\in S$, outputs $f(q)$.
\item A real number $x$  is \emph{computable} if  there exists a TM $\cM$ which, upon input $n\in\NN$, outputs
  $q_n\in\QQ$ such that
  $$|x-q_n|<2^{-n},$$
  computable elements of $\CC$ and $\RR^d$ are similarly defined.

\item Turing \cite{Tur} also produced an example of a non-computable real number  based on the famous {\it Halting problem} which can be formulated as follows: {\it verify from the text of a computer program, whether this program will terminate, or will execute indefinitely}. Turing showed in \cite{Tur}  that the Halting problem is algorithmically unsolvable (that is, a TM which realizes such verification cannot exist). Enumerating all algorithms in some way as $\bset{\cM_n}$ (for instance, listing all possible finite combinations of characters allowed by a programming language in lexicographic order), we can see that the number
\begin{equation}\label{alpha}
\alpha=\sum 3^{-n}p(n),\text{ where }p(n)=1\text{ if }\cM_n\text{ halts, and }p(n)=0\text{ otherwise }
\end{equation}
is not computable.

\item The notion of an argument of a function of a real variable is formalized as follows: a function $\phi:\NN\to \QQ$
  is an {\it oracle} for a number $x\in \RR$ if
  $$|\phi(n)-x|<2^{-n}.$$
  An \emph{oracle TM} $\cM^\phi$ can query a value of $\phi(n)$ for any $n$. Oracles for points in $\RR^d$ or $\CC$ are defined in a similar fashion.
\item A function $f$ from a set $S\subset \RR^d$ to $\RR^m$ is \emph{computable} if there exists a TM $\cM^\phi$ with an oracle for $x\in S$ which, upon an input $n$, outputs $q_n\in\QQ^m$ such that
  $$||q_n-f(x)||<2^{-n}.$$
  This means, in particular, that for a given $x\in S$ and $\eps=2^{-n}$, there is an algorithm to find  $\delta(\eps)=2^{-k}$ such that for $||s-x||<\delta(\eps)$ we have $||f(s)-f(x)||<\eps$  (here $k$ is the largest value for which $\cM^\phi$ queries $\phi(k)$ to output $q_n$). In particular, \emph{all computable functions are continuous on their domains of definition}. The definition of a computable function extends to computable metric spaces (see \cite{BBRY}).

%

  \item A sequence of computable objects $\mathcal{O}_1,\mathcal{O}_2,\dots$ is \defin{uniformly computable} if there is a single Turing Machine that on input $(n,m)$, computes an approximation of $\mathcal{O}_n$ with the error bounded by $2^{-m}$.

  \item An open set $U \subset \R^d$ is \emph{lower computable} if there is a uniformly computable sequence of rational balls $(B_i)_i$ whose union exhausts $U$.

\item Finally, a measure $\mu$ on a Borel set in $\RR^d$ is a \emph{computable measure} if for every rational cube $C\subset \RR^d$ and  every sequence $(f_{j}:\RR^d\to\RR)_{j\in\NN}$ of uniformly computable functions in $\RR^d$ there exists a TM $\cM$ which on input $(j,n)\in\NN^2$ outputs a rational $I_{j,n}$ satisfying
$$
|I_{j,n} - \int_C f_{j}\,d\mu | < 2^{-n}.
$$
In other words, a computable measure can be algorithmically approximated in the weak sense. We note that restricting functions $f$ to computable $1$-Lipschitz functions (i.e. the functions satisfying $|f(x)-f(y)|\leq \|x-y\|$) produces an equivalent definition.
\end{itemize}

When an object is computable by a TM $M^\phi$ with an oracle for $x$, we will say that it is \emph{computable relative to }$x$. For instance, a computable function $f:S\to \RR^m$ is computable at every $x\in S$ relative to $x$. Moreover, since the computation for every $x\in S$ is performed by the same oracle TM $M^\phi$ with an oracle for $x$, we say that the function $f$ is {\it uniformly} computable relative to $x\in S$. On the other hand, a function which is computable at every $x\in S$ relative to $x$ does not need to be a computable function $S\to\RR^m$ since the computation may need to be performed by different algorithms for different pieces in $S$.

A trivial example is given by any discontinuous function. For instance, if $f=\operatorname{sign}(x)$ is the standard sign function on $\RR$ which assumes only three values: $-1$, $0$, $+1$, then $f$ is computable relative to $x$ at every point $x\in \RR$. However, it is not a computable function $\RR\to\RR$ since it is only piecewise continuous. Intuitively, when an algorithm reads a sequence of finite rational  approximations of $x$ one by one  through an external input and keeps seeing zeros, it has no way of knowing whether $x$ is truly a zero, or just a very small positive or negative number. The algorithm has to output a value after a finite number of queries of the oracle, so it will be left guessing (sometimes wrongly). But if we exclude $0$ from the domain of the sign function, then it becomes uniformly computable -- the algorithm will just need to query the rational approximations of $x$ through the  oracle long enough  to get a determination of the sign of $x$ in this case.

It is much harder to give an example of a continuous non-computable function, which is nonetheless computable relative to $x$ at every $x$ of its domain (just not by the same algorithm). Surprisingly, as we will see below, such examples may arise naturally in our context.

In view of the above discussion,
given a domain $\Omega$ and $x\in\Omega$, we say that $\Omega$  has a \emph{ computable harmonic measure at $x$} if $\omega_x^\Omega$ is computable relative to $x$. In case there is a single machine with an oracle for $x$ that computes $\omega_x^\Omega$ for every $x$ (in other words, $\omega^\Omega_x$ is uniformly computable relative to $x$), we will say that $\Omega$ has a \emph{uniformly computable harmonic measure}.

 We say that a function $$\tau:\Omega\times \Omega\to \RR_{>0}$$ is a \emph{Harnack bound} for the domain $\Omega$ if, for any positive harmonic function $u$ and every pair of points $z,\ w\in\Omega$, we have
$$\frac{1}{\tau(z,w)}\leq\frac{u(z)}{u(w)}\leq \tau(z,w).$$
A {\sl computable Harnack bound}  is a computable function $\tau$ with this property.


We say $\Omega$ is {\it computably regular} if there exists a computable positive function $\eps(n)\,:\,\NN\mapsto\QQ$ so that
$$
\dist(x,\partial\Omega)<\eps(n)\Rightarrow \omega_x^\Omega(B(x,2^{-n}))>1-2^{-n}.
$$
Observe that any computably regular domain is regular. It is an easy consequence of Beurling Projection Theorem that any simply connected planar domain is computably regular.

\subsection{Overview of the results}

All the domains considered in this paper are assumed to be bounded.

Let us start by recalling that in \cite[\S 5.2]{BBRY}, the authors give an example of a domain $\Omega$ for which $\partial\Omega$ is computable in the strongest possible sense (see Section \ref{approximations}), but whose harmonic measure is not computable. It follows that the computability of these objects must involve data about $\Omega$ which is from a somewhat different nature.

Our first result says that if at one point $x_0$ there exists an algorithm to compute $\omega_{x_0}^\Omega$, then at any other point $x\in\Omega$ such an algorithm  also exists. Note that no assumptions about $\Omega$ are made.  The proof is based on the notion of {\it computable harmonic approximations} introduced below, which characterizes the local computability of the measure  (see Theorem \ref{lem:harm_measure_iff_approx}).

\medskip
\begin{mainthm}\label{thm:allcompute}
 Let $\Omega$ be a connected domain and let $x_0\in \Omega$. If $\Omega$ has a computable harmonic measure at $x_0$, then $\Omega$ has a computable harmonic measure at $x$, for every $x\in \Omega$.
\end{mainthm}

\noindent
 Interestingly, as the following example shows, the computability of harmonic measure in the above result does not need to be uniform, and different machines may be required depending on the point $x$.

\begin{mainthm}\label{thm:nocompute_measure}
There exists a regular domain $\Omega$ such that:
\begin{itemize}
\item[{i)}]  $\Omega$ has a computable harmonic measure at $x$, for every $x\in \Omega$;
\item[{ii)}] The harmonic measure of $\Omega$ is not uniformly computable.
\end{itemize}
\end{mainthm}

\medskip
\noindent
As a remarkable consequence of this example we obtain the following:

\begin{mainthm}\label{thm:nocompute_Dirichlet}
There exists a regular domain $\Omega$ and a computable simple function $f: \partial\Omega \to \R$ such that the unique solution $u(x):\Omega \to \R$ to the Dirichlet problem with the boundary data given by $f$ satisfies:
\begin{itemize}
\item[i)] For each $x\in \Omega$ the value $u(x)$ is computable relative to $x$;
\item[ii)] $u$ is not a computable function.
\end{itemize}
\end{mainthm}

As demonstrated in the next Theorem, this phenomenon cannot happen for lower computable domains.

\begin{mainthm}\label{thm:uni_comp}
Let $\Omega$ be a lower computable domain which has a computable harmonic measure at $x_0$ for some computable $x_0\in\Omega$. Then  $\Omega$ has uniformly computable harmonic measures.
\end{mainthm}

We do not know if {\it every} domain with uniformly computable harmonic measures is lower computable. This is an interesting question, which is open even when $\Omega$ is a planar domain.


Let us now return to the original question: what geometric conditions would guarantee the computability of the harmonic measure of a domain at a computable point? For regular domains with computable boundaries, there is a necessary and sufficient condition.

\begin{mainthm}\label{thm:computable_boundary}
Let $\Omega$ be a regular domain with a computable boundary. Then the following are equivalent:
\begin{enumerate}
\item $\Omega$ is computably regular.
\item $\Omega$ has a computable harmonic measure at $x_0$ for some computable $x_0\in\Omega$.
\item $\Omega$ has uniformly computable harmonic measures.
\end{enumerate}
\end{mainthm}

Note that the condition of regularity is important here. For example, $\DD\setminus{0}$ is a domain with a computable boundary and a uniformly computable harmonic measure (which agrees with the harmonic measure of the unit disk). However, it is not regular, so it is not computably regular.

In the more useful class of lower-computable domains, there are regular domains with a uniformly computable harmonic measure which are not computably regular (see Lemma \ref{lem:non_comp_regular}). However, if we just assume the computable regularity \emph{a prioiri}, we get a complete geometric characterization of domains with a computable harmonic measure.

\begin{mainthm}\label{thm:lower_computable}
Let $\Omega$ be a computably regular domain. Then the following are equivalent:
\begin{enumerate}
\item $\Omega$ and $\partial\Omega$ are both lower computable.
\item $\Omega$ has a computable harmonic measure at $x_0$ for some computable $x_0\in\Omega$.
\item $\Omega$ has uniformly computable harmonic measures.
\end{enumerate}
\end{mainthm}

\section{Preliminaries}
\subsection{Approximations of domains}\label{approximations}

We start by briefly recalling the standard notions of computability for sets in $\RR^d$.

We will partition $\RR^d$ into dyadic cubes of the form $$\prod_{l=1}^d\left[k_l2^{-m},\, (k_l+1)2^{-m}\right)\text{, where }m,\,k_1,\dots k_d\in\ZZ.$$
For brevity, we will refer to a connected interior of a finite union of dyadic cubes of the same size $2^{-m}$ as  a \emph{dyadic polygon of rank $m$}. Note that every dyadic polygon of rank $m_0$ is also a dyadic polygon of rank $m$ for any $m\geq m_0$. We do not assume that the polygons are convex or even simply connected: polygonal holes inside are allowed. Thus a union of two dyadic polygons with intersecting interiors is again a dyadic polygon (with the rank equal to the maximum of two ranks of the polygons).

Recall that a sequence of dyadic polygons $\bset{P_n}$ is  uniformly computable if there exists a TM $\cM$ which, upon input $n\in\NN$, outputs the size $2^{-m}$ and the finite collection of $d$-tuples $(k_1,\dots k_d)$, which constitutes a finite description of $P_n$.

Observe that for any dyadic polygon $P$ its harmonic measure $\omega_x^P$ is uniformly computable. One of the ways to see it is to use the algorithm described in \cite{BB07} (see Proposition 5.6 of \cite{BBRY} for the details.)

Lower computability of an open set $\Omega\subset\RR^d$ (defined in Section \ref{stage}) is equivalent to the existence of a uniformly computable sequence of dyadic polygons $\bset{P_n}$  such that
 \[
 \Omega=\bigcup_n P_n.
 \]
 A closed set $E\subset \RR^d$ is \emph{upper computable} if its complement $\RR^d\setminus E$ is a lower computable open set; $E$ is \emph{lower computable} if there is a uniformly computable sequence of dyadic polygons
 $\bset{P_n}$ such that
 $E$ intersects the interior of a dyadic polygon $P$ if an only if $P=P_{n}$ for some $n\in\NN$.

 A closed set $E$ is called \emph{computable} if it is simultaneously lower and upper computable.
 \begin{prop}
   A compact set $K\Subset \RR^d$ is computable if and only if there exists a  uniformly computable sequence of finite sets $Q_n\subset \QQ^d$ such that the Hausdorff distance
   $$\dist_H(Q_n,K)<2^{-n}\,\,\text{ for all }n.$$
 \end{prop}
 \begin{proof}
   Suppose $K\Subset\RR^d$ is simultaneously lower and upper computable. Let $C\subset \RR^d$ be a large enough closed rational cube
   compactly containing $K$.
      Consider two TMs running in parallel:
\begin{itemize}
\item $\cM_1$ computing a sequence of rational polygons $\bset{P_n}$ which exhaust the open set $\RR^d\setminus K$;
  \item $\cM_2$ computing a sequence of all rational cubes $C_n$ of size $2^{-(n+1)}$ which intersect $K$.
 \end{itemize}
  When the Hausdorff distance between
  $$S^j_1=\text{cl} \left(C\setminus\bigcup_{m\leq j}P_m\right)\text{ and }S^j_2=\bigcup_{m\leq j} C_m$$ is less than $2^{-(n+1)}$, the union of centers of the cubes
  $C_m$, $m\leq j$ forms the desired set $Q_n$.

  To prove the converse, we use the TM uniformly computing sets $Q_n$ to enumerate all dyadic polygons $P$ such that there exists
  $n\in\NN$ and $y\in Q_n$ for which the cube $C(y,2^{-(n-1)})$ with side $2^{-(n-1)}$ and center $y$ lies in $P$. This enumerates all polygons which intersect $K$, thus proving that it is lower computable. For upper computability, we use the TM computing the sets $Q_n$ to output
  the polygons $$P_n=\RR^d\setminus \bigcup_{y\in Q_n}\text{int}(C(y,2^{-(n-1)})),$$
  which exhaust the complement of $K$.
 \end{proof}

Let $\Omega$ be a domain in $\RR^d$ and $x\in\Omega$.

Our next definition is a quantitative version of approximating $\Omega$ by a sequence of polygons. Namely,
we say that a sequence of polygons $\bset{P_n}$ is an {\it interior approximation} of $(\Omega,x)$ if
\begin{enumerate}
\item $x\in P_n\Subset\Omega$ for all $n$.
\item For every $y\in\partial P_n$ we have $\dist(y,\partial\Omega)<2^{-n}$.
\end{enumerate}
We say $(\Omega,x)$ has a {\it computable interior approximation} if there is a uniformly computable sequence of dyadic polygons $\bset{P_n}$ which is an interior approximation of $(\Omega,x)$.

\begin{rmk}
  \label{rem-approx}
Note that if $\bset{P_n}$ is an interior approximation of $(\Omega,x_0)$, then $\bset{P_n}_{n\geq n_0}$ is also an interior approximation for $(\Omega,x)$ for every $x\in \Omega$, where $n_0=n_0(x)$ is large enough.
Also, evidently, if there exists a computable interior approximation of $(\Omega,x)$ for some $x\in\Omega$, then $\Omega$ is lower computable.
\end{rmk}

In this paper, it will be convenient for us to work with $C^2$-smooth subharmonic functions, i.e. $C^2$-smooth functions with non-negative Laplacian.

The following lemma shows that in order to conclude that the harmonic measure is computable, it is enough to be able to integrate every ``good'' subharmonic function with respect to the harmonic measure.
\begin{lem}\label{lem:computable_decomposition}
Let $\mu$ be a measure supported inside $[0,1]^d$.
Then $\mu$ is a computable measure if and only if
for every sequence $\{f_{j}:\RR^d\to\RR\}_{j\in\NN}$ of uniformly computable positive $C^2$-smooth subharmonic functions in $[0,1]^d$, bounded by $1$ and $1$-Lipshitz, there exists a TM $\cM$ which on input $(j,n)\in\NN^2$ outputs a rational $I_{j,n}$ satisfying
$$
|I_{j,n} - \binr\mu{f_{j}} | < 2^{-n}.
$$
We use the notation
$$
\binr \mu f:=\int f(x)\;d\mu(x).
$$
\end{lem}
\begin{proof}
It is enough to only prove the ``if'' direction.

Let  $\bset{g_j}$ be a uniformly computable sequence of $1$-Lipschitz functions in $[0,1]^d$.  Let $\varphi(x)$ be a positive computable $C^\infty$-smooth function supported inside  $B(0,1)$ with $\int\phi(x)=1$. By convolution with the functions
$\varphi_n(x):=2^{n}\varphi(2^nx)$ we create a uniformly computable sequence of $C^\infty$-smooth functions
$$f_j^n:=g_j*\varphi_n,$$
such that
\begin{equation}\label{eq:supbound}
\|f_j^n-g_j\|_\infty\leq 2^{-n}.
\end{equation}
All of these functions are compactly supported inside a larger cube $(-2,2)^d$.

 Let $D_j^n$ denote the Laplacian of $f_j^n$. Following \cite{Zhong1998} we see that $D_j^n$ are uniformly computable functions. Let $g$ denote Green's function for the cube $[-2,2]^d$. This function is explicitly computable. Define two sequences by
\begin{eqnarray*}
u_j^n(x)&:=&\integrate {[-2,2]^d}{}{D_{j+}^n(y)\cdot g(x,y)}m(y),\;\\v_j^n(x)&:=&\integrate {[-2,2]^d}{}{D_{j-}^n(y)\cdot g(x,y)}m(y),
\end{eqnarray*}
where
$
D_{j+}^n(x)=\max\bset{D_j^n(x),0}\;, D_{j-}^n(x):=\max\bset{-D_j^n(x),0}
$. Thus $D_j^n(x)=D_{j+}^n(x)-D_{j-}^n(x)$.

We first note that $u_j^n$ and $v_j^n$ are uniformly computable as integrals of uniformly computable functions over computable sets with respect to Lebegue's measure, which is computable. By Green's identities,
\begin{equation}\label{eq:Green_decompose}
f_j^n(x)=\integrate {[-2,2]^d}{}{f_j^n(y)\cdot\Delta g(x,y)}m(y)=u_j^n(x)-v_j^n(x)
\end{equation}
since $f_j^n$ are compactly supported in $(-2,2)^d$.

Also by Green identities, the Laplacians of $u_j^n$ and $v_j^n$ are equal to $D_{j+}^n$ and $D_{j-}^n$ respectively. Thus they are subharmonic and $C^2$-smooth. Since they are uniformly computable, there are uniformly computable constants $S_j^n$, such that
$$\widetilde{u}_j^n:=u_j^n/S_j^n,\quad \widetilde{v}_j^n:=v_j^n/S_j^n$$ are $1$-Lipshitz and bounded by $1$.

Using the algorithm for computing the sequence $\{g_j\}$, we can uniformly compute the sequences of the subharmonic functions $\{\widetilde{u}_j^n\}$ and $\{\widetilde{v}_j^n\}$. Using the hypothesis of the Lemma, we can uniformly compute $2^{-n}-$approximations of
$$\int f_j^n\,d\mu=S_j^n\left(\int \widetilde{u}_j^n\,d\mu-\int \widetilde{v}_j^n\,d\mu\right).$$
By \eqref{eq:supbound}, these are $2^{-n+1}-$approximations of $\int g_j\,d\mu$.
\end{proof}

\begin{rmk}\label{rmk:decompose}
Equation \eqref{eq:Green_decompose} also implies the well-known fact that any (not necessarily computable) $C^2$-smooth function can be represented as a difference of two positive smooth subharmonic functions.
\end{rmk}

We will make extensive use of the classical notion of \emph{harmonic correction} (see \cite{Subharmonic}):
\begin{defn}\label{def:harm_cor}
Let $u$ be a function in $[0,1]^d$ and let $\Omega\subset [0,1]^d$. Define the harmonic correction of $u$ in $\Omega$ as
$$u^{\Omega}(x):=
\begin{cases}
u(x), & x\notin\Omega\\
\binr{\omega_x^{\Omega}}u& x\in\Omega
\end{cases},$$
\end{defn}
Let us list the standard properties of harmonic corrections of subharmonic function, see \cite{Subharmonic} for the details.
\begin{prop}\label{prop:subharmonic}
Let $u$ be subharmonic in $[0,1]^d$ and $\Omega\subset [0,1]^d$ be a connected domain. Then the harmonic correction $u^{\Omega}$ satisfies the following properties:
\begin{enumerate}
\item If $\Omega$ is regular, then $u^\Omega$ is harmonic on $\Omega$ and subharmonic on $[0,1]^d$.
\item If $\Omega'\subset\Omega$, then $u(x)\leq u^{\Omega'}(x)\leq u^{\Omega}(x)$.
\item If $v$ is harmonic in $\Omega$, and $v(x)\geq u(x)$ for $x\in\Omega$, then $v(x)\geq u^{\Omega}(x)$ for $x\in\Omega$.
\end{enumerate}
\end{prop}
The next proposition shows that any interior approximation can be used to approximate the harmonic measure in the weak sense. Its proof is the same as Wiener's solution for Dirichlet's problem. We include it for completeness.
\begin{prop}\label{lem:approx_is_approx}
Let $\Omega$ be a bounded connected domain and let $\bset{P_n}$ be an  interior approximation of $\Omega$. Then for every $x\in \Omega$ and every $f$ continuous in a neighbourhood of $\partial\Omega$
\begin{equation}\label{eq:convergence}
\binr{\omega_x^n}f\underset{n\rightarrow\infty}\longrightarrow \binr{\omega_x}f,
\end{equation}
where $\omega_x:=\omega_x^\Omega, \; \omega_x^n:=\omega_x^{P_n}$.
\end{prop}
\begin{proof}
By the density of smooth functions and Remark \ref{rmk:decompose}, it is enough to prove the result for smooth subharmonic functions $f$. By weak$^*$-compactness of probability measures, it is enough to assume that $P_n\subset P_{n+1}$.

Observe that for $x\in P_n$,
$$\binr{\omega_x^n}f=f^{P_n}(x),\quad \binr{\omega_x}f=f^{\Omega}(x).$$
Thus, by Proposition \ref{prop:subharmonic}
$$f(x)\leq\binr{\omega_x^n}f\leq\binr{\omega_x^{n+1}}f\leq\binr{\omega_x}f.$$
So
$$v(x):=\lim_{n\to\infty}\binr{\omega_x^n}f.$$
is a limit of increasing sequence of bounded above by a harmonic function $f^{\Omega}(x)$ functions which are harmonic in $P_n$. So $v$ itself is harmonic in $\Omega=\cup P_n$, and
$$f(x)\leq v(x)\leq f^{\Omega}(x)$$
By Proposition \ref{prop:subharmonic}(3):
$$\lim_{n\to\infty}\binr{\omega_x^n}f=v(x)=f^{\Omega}(x)=\binr{\omega_x}f.$$
\end{proof}

Note that the existence of computable interior approximations does not necessarily imply computability of harmonic measure, even relative to a single computable point in $\Omega$. As is  shown in \cite[\S 5.2]{BBRY}, it may be the case that the convergence in (\ref{eq:convergence}) above does not occur at a computable rate, so that the harmonic measure $\omega_x^\Omega$ is not computable.
In order to give a condition for  the computability of the harmonic measure relative to a point, we introduce the following notion. Let $x\in\Omega$ and $Q\subset\Omega$ be a dyadic polygon, possibly empty. We say that a sequence of polygons $\bset{P_n}$ is a {\it harmonic approximation} of $(\Omega,x,Q)$, if:
\begin{enumerate}
\item $x\in P_n$ and $Q\subset P_n$ for all $n$.
\item For every 1-Lipschitz $C^2$-smooth subharmonic function $f$ which is bounded by 1,
$$
\abs{\binr{\omega_{x}^{P_n}}{f}-\binr{\omega_{x}^\Omega}{f}}=\abs{f^{\Omega}(x)-f^{P_n}(x)}<2^{-n}.
$$
\item  $\omega_x^\Omega(\partial \Omega\cap \overline{P_n})<2^{-n}$.
\end{enumerate}
We say that $(\Omega,x, Q)$ has a {\it computable harmonic approximation} if there exists a uniformly computable sequence of dyadic polygons $\bset{P_n}$ which is a harmonic approximation of $(\Omega,x,Q)$.

We then have:
\begin{thm}
  \label{lem:harm_measure_iff_approx}
  Let $\Omega$ be a connected domain in $[0,1]^d$. The  harmonic measure of $\Omega$  is computable at $x$ if and only if $(\Omega,x,Q)$ has a computable harmonic approximation for any dyadic polygon $Q\subset\Omega$.
\end{thm}

\begin{rmk}\label{rmk:empty_enough}
The proof actually shows that the computability of the  harmonic measure of $\Omega$ at $x$ is equivalent to the existence of computable harmonic approximation for an empty $Q$.
\end{rmk}

\begin{proof}
Assume that the harmonic measure is computable at $x$.
 Denote the collection of all dyadic polygons of rank $n$ containing $x$ and $Q$  as $\mathcal P_n$.  We define the partial order on $\mathcal P_n$ by $P\le P'$ if $P\subseteq P'$. Define $P_n=\bigcup_{P\Subset\Omega} P$. By maximality of $P_n$, for every $y\in\partial P_n$ we have $\dist(y,\partial\Omega)<d\cdot 2^{-n}$.

Although not necessarily computable, the sequence $\bset{P_n}$ is an interior approximation, and following Proposition \ref{lem:approx_is_approx} we know that for every $f$ continuous in a neighbourhood of $\partial\Omega$
\begin{equation}\label{eq:dens}
\binr{\omega_x^{P_n}}f\underset{n\rightarrow\infty}\longrightarrow \binr{\omega^\Omega_x}f.
\end{equation}
In order to compute the desired harmonic approximation of $\Omega$, we will search among all the polygons in $\mathcal{P}=\cup_n \mathcal{P}_n$. Let $R_1,R_2,\dots$ be an algorithmic enumeration of all the polygons in $\mathcal{P}$.  Fix $n$ and let $\bset{f_1,\cdots, f_{m_n}}$ be a set of computable subharmonic $C^2$-smooth 1-Lipschitz functions bounded by 1 so that for every subharmonic $C^2$-smooth 1-Lipschitz function $f$ bounded by 1 there exists $j$ so that
\begin{equation}\label{eq:dens_lip}
\underset{y\in[0,1]^d}\sup\;\abs{f(y)-f_j(y)}<2^{-n-2}.
\end{equation}
Note that one can choose the functions in the collection $\bset{f_\ell}_{\ell=1}^{m_n}$ to be computable uniformly in $n$. We can use a greedy search to find a polygon $Q_n=R_{k_n}$ such that
\begin{enumerate}
\item $x$ lies in the interior of $Q_n$, $Q\subset Q_n$.
\item For every $j=1,\dots ,m_n$
$$
\abs{\binr{\omega_{x}^{Q_n}}{f_j}-\binr{\omega^\Omega_{x}}{f_j}}<2^{-n-1}.
$$
\item $\omega_x^\Omega(\partial \Omega\cap \overline{Q_n})<2^{-n-1}$.
\end{enumerate}
Indeed, the first point holds by the definition of $\mathcal P_n$. The other two points can be tested since both harmonic measures are computable. The fact that these three points must hold for some $Q_n$, and thus the searching algorithm eventually halts, follows from \eqref{eq:dens}. In particular, the second point and \eqref{eq:dens_lip} imply that for every subharmonic smooth 1-Lipschitz function $f$ bounded by 1, we have
$
\abs{\binr{\omega_{x}^{Q_n}}{f}-\binr{\omega^\Omega_{x}}{f}}<2^{-n}.
$
We see that the collection $\bset{Q_n}$ is a computable harmonic  approximation at $x$.

Now assume there exists a computable harmonic approximation $\bset{P_n}$ for $(\Omega,x,\emptyset)$. Let  $\bset{g_j}$ be a uniformly
computable sequence of subharmonic $C^2$-smooth $1$-Lipschitz functions in $[0,1]^d$ bounded by $1$.
By property 2 of the approximation $\bset{P_n}$, we know that
$$
\abs{\binr{\omega_{x}}{g_j}-\binr{\omega_{x}^n}{g_j}}<2^{-n}.
$$

Let $\cM_\ell$ be the Turing Machine which realizes the computability of  $\omega_x^{P_{k_\ell}}$, and denote $I^\ell_{j,n}$ its sequence of outputs, which approximates $<\omega_x^{{\ell}},g_j>$ with precision $2^{-n}$. Then
the uniformly computable sequence $I^n_{j,n+1}$ approximates $<\omega_x,g_j>$ with precision $2^{-n}$, which means the harmonic measure $\omega_x$ is indeed computable.
\end{proof}

\subsection{Some auxiliary results}\label{aux}
We begin by stating:
\begin{lem}\label{lem:inner_computable_dist}
Let $\Omega\subset[0,1]^d\subset\R^d$ be a lower computable connected domain with a lower computable boundary. Then the function $d:\Omega\rightarrow\R$ defined by $d(x)=\dist(x,\partial\Omega)$ is computable.
\end{lem}
Note that we may not be able to compute the distance function for points outside of $\Omega$.
The proof of the lemma is straightforward: since we can computably exhaust the interior of $\Omega$ with rational polygons, and computably enumerate all rational polygons which intersect $\partial\Omega$, for each $x\in\Omega$ we are able, by a greedy search, to find the nearest ball of radius $2^{-n}$ which intersects $\partial\Omega$. The details are left to the reader.

\begin{cor}\label{cor:compute_inner}
Let $\Omega\subset[0,1]^d\subset\R^d$ be a lower computable connected domain with a lower computable boundary and $x\in\Omega$ be a computable point. Then $(\Omega, x)$ has a computable interior approximation.
\end{cor}
\begin{proof}
Let $\{P_k\}$ be a uniformly computable sequence of dyadic polygons such that $\Omega=\cup_k P_k$. We can define a uniformly computable increasing sequence of dyadic polygons by
$$R_\ell:=\cup_{m=1}^\ell P_{k_m}, $$
where
$$k_1:=\min\{k\,:\,x\in P_k\},\quad k_m=\min\{k\,:\,R_{m-1}\cap P_k\neq\emptyset\}.$$
Since $\cup_\ell R_\ell=\Omega$ and $\partial\Omega$ is compact
$$\lim_{\ell\to\infty}\sup_{y\in\partial R_\ell }\dist(y,\partial\Omega)=0.$$
Thus we can use Lemma \ref{lem:inner_computable_dist} to compute a sequence $\ell_n$ so that for any vertex $y$ of $\partial R_{\ell_n}$, considered as a dyadic polygon of rank at least $n+1$, we have
$\dist(y,\partial\Omega)<\frac{2^{-n}}d$.
The sequence $R_{\ell_n}$ is the desired interior approximation.
\end{proof}

We will also need the following result.
\begin{lem}\label{lem:lower_computable_boundary}
Let $\Omega$ be a domain and $x\in\Omega$. Assume that the harmonic measure $\omega_x^\Omega$ is computable relative to $x$. Then $\partial\Omega$ is lower computable.
\end{lem}
\begin{proof}
Let us enumerate all the dyadic polygons in $[0,\,1]^d$. Let us define a uniformly computable sequence
$$f_k(y):=\chi_{P_k}(y)\times\dist(y,\partial P_k).$$
Then we can use the algorithm uniformly computing $\binr{\omega_x^\Omega}{P_k}$ to decide for every pair $(P_k,n)$ whether $\omega_x^\Omega(f_k)>2^{-n}$. Since $\omega_x^\Omega(f_k)>0$ if and only if $P\cap\partial\Omega\neq\emptyset$, this process can be used to enumerate all dyadic polygons intersecting $\partial\Omega$.
\end{proof}

We also need the following estimate on Harnack bound
\begin{lem}\label{uniform_harnack_bound}
Let $\Omega\subset[0,1]^d\subset\R^d$ be a domain and $\gamma\subset\Omega$ be a closed connected set such that for any $y\in\gamma$, $\dist(y,\partial\Omega)>2^{-n}$. Then for any $x_1,\ x_2\in\gamma$ we have the following uniform Harnack bound
$$\tau(x_1, x_2)\leq C(n,d):=C(d)^{2^{nd}}$$
where $C(d)$ is a computable function of $d$.
\end{lem}
\begin{proof}
Let us take $\ell$ such that $2^\ell>d$ and consider the dyadic polygon $P$ consisting of all dyadic cubes of rank $n+l+1$  intersecting $\gamma$. Let $Q\subset P$ be a dyadic cube of rank $n+l+1$. Since distance from this cube to $\partial\Omega$ is bounded bellow by $2^{-n+1}$, there is some explicitly computable $C(d)$ such that $\tau(y_1,y_2)\leq C(d)$ for any $y_1,\ y_2\in Q$. Since $\gamma\subset P$ and $P$ consists of at most $2^{nd}$ cubes, we can apply the previous estimate repeatedly to get the desired bound.
\end{proof}
We will need the following two standard results for our constructions in Section \ref{sec:negative}.
\begin{lem}\label{observation}Let $\Omega\subset B(a,R)\subset\CC$ be a connected subdomain of some disk of radius $R$ in $\CC$ and let $K\Subset \partial\Omega$ be a set of diameter $d\leq R/2$. Then for every $x\in\Omega$,
$$
\omega_x^{\Omega}(K)\le\frac{\log\bb{\frac{2R}{\dist(x,K)}}}{\log R/d}.
$$
\end{lem}
\begin{proof}
Let $x_0\in K$ . Define the function $u(\zeta):=\frac{\log\bb{\frac{2R}{\abs{\zeta-x_0}}}}{\log R/d}$.
This function is harmonic in $\Omega$,  non-negative in $B(a,R)$ and for every $\zeta\in\partial K$ we know that
$$u(\zeta)\geq\frac{\log\bb{\frac{2R}{d}}}{\log R/d}\geq1.$$ Thus, by the maximum principle, for every $\zeta\in \Omega$:
$$
\omega_\zeta^{\Omega} (K) \le u(\zeta)
$$
which is the desired inequality.
\end{proof}

\begin{lem}\label{lem:compare_harmonic}
Assume that a domain $\Omega$ is a subdomain of a domain $\Omega'$. Then for any $x\in\Omega$ and for any continuous function $f$ bounded by $1$ we have
$$\abs{\binr{\omega_{x}^{\Omega}}{f}-\binr{\omega^{\Omega'}_{x}}{f}}\leq2\omega_x^{\Omega}(\partial\Omega\cap\Omega').$$
\end{lem}
\begin{proof}
Observe that the function
$$h(x):=\binr{\omega_{x}^{\Omega}}{f}-\binr{\omega^{\Omega'}_{x}}{f}$$
is harmonic in $\Omega$, tends to zero at regular points of $\partial\Omega\cap\partial\Omega'$ and is bounded between $-2$ and $2$. Thus, by the maximum principle
$$-2\omega_x^{\Omega}(\partial\Omega\cap\Omega')\leq h(x)\leq2\omega_x^{\Omega}(\partial\Omega\cap\Omega'), \quad x\in\Omega$$
\end{proof}

\section{Proofs}\label{sec:proofs}

\subsection{Proofs of Theorems \ref{thm:allcompute} and \ref{thm:uni_comp}}
Let us first show:
\begin{lem}
  \label{lem:harnack1}
  There exists an algorithm with oracles for points $x_0$, $x$ in $\Omega$ which takes as an input the finite description of a dyadic polygon  $Q$ of some rank $\ell$ such that the following holds. Assume that
  \begin{enumerate}
  \item the harmonic measure of $\Omega$ at $x_0$ is computable;
  \item $Q\subset \Omega$;
  \item there is a closed connected set $\gamma\subset Q$, $x\in\gamma$, $x_0\in\gamma$ with the following property: if $y\in\gamma$, then $\dist(y,\partial Q)>d2^{1-\ell}$.
  \end{enumerate}
  Then the algorithm computes the harmonic measure of $\Omega$ at $x$.
\end{lem}
\begin{proof}

  Using Theorem \ref{lem:harm_measure_iff_approx}, we can find a computable harmonic approximation $\bset{P_n}$ of $(\Omega,x_0, Q)$ such that
  $Q\subset P_{n}$ for all $n$.

We can apply Lemma \ref{uniform_harnack_bound} to compute a Harnack bound $C=C(n,d)>0$ between $x$ and $x_0$ in $Q$.

Compute $k_n>n+1$ large enough so that $2^{-n}>2^{2-k_n}\cdot C.$
Let us show that $Q_n:=P_{k_n}$ is a computable harmonic approximation of $(\Omega, x, Q)$. By Theorem \ref{lem:harm_measure_iff_approx} and Remark \ref{rmk:empty_enough}, the harmonic measure at $x$ is also computable.

The first property of the harmonic approximation holds automatically, since $$x\in Q\subset P_{k_n}=Q_{n}.$$

To check  the third property, use the Harnack bound $C$ for the positive and harmonic in $Q$ function
$$u(y):=\omega^\Omega_y(\partial\Omega\cap \overline{P_{k_n}})$$
to get
\begin{equation}\label{eq:harm_bound}\omega^\Omega_x(\partial\Omega\cap \overline{P_{k_n}})\leq C\omega^\Omega_{x_0}(\partial\Omega\cap \overline{{P_{k_n}}})<C2^{-k_n}<2^{-n-1}.
\end{equation}

To check the second property, fix a positive subharmonic $C^2$-smooth $f$ bounded by $1$. Then, the function
$$w(y):=f^{Q_n}(y)-f^\Omega(y)$$
is subharmonic on $\Omega$, as a difference of a subharmonic in $\Omega$ function $f^{Q_n}(y)$ and a harmonic function $f^{\Omega}(y)$.

Observe that $w(y)=f(y)-f(y)=0$ for any regular point $y\in\partial\Omega$ which is  not in $Q_n$. Also $w(y)\leq1$. Thus, by maximum principle, $w(y)\leq u(y)$.

Now we can define a positive function
$$v(y):=f^{\Omega}(y)-f^{Q_n}(y)+u(y)=u(y)-w(y)\geq0$$

Since all three functions defining $v$ are harmonic on  $Q$, so is $v$. Now we can use the Harnack bound $C$ and the inequality \eqref{eq:harm_bound} to get
\begin{multline}-2^{-n}<-u(x)\leq\left(f^\Omega(x)-f^{Q_n}(x)\right)\leq v(x)\leq Cv(x_0)=\\ C\left(f^\Omega(x_0)-f^{Q_n}(x_0)+u(x_0)\right)<C(2^{-k_n}+2^{1-k_n})<2^{-n}\end{multline}
which is exactly the second property.

\end{proof}

\begin{proof}[Proof of Theorem \ref{thm:allcompute}]
Since a dyadic polygon $Q\subset\Omega$ satisfying condition 3 of the previous Lemma always exists for any two points $x_0,\, x\in\Omega $, Theorem \ref{thm:allcompute} immediately follows from Lemma \ref{lem:harnack1}. Note that the computability of the harmonic measure $\omega_x^\Omega$ is not uniform relative to $x$ since the algorithm for its computation requires the input of such a polygon as a parameter.
 \end{proof}

\begin{proof}[Proof of Theorem \ref{thm:uni_comp}]
Since $\Omega$ is lower computable, for any $x$ we can compute a dyadic polygon $P$ of rank $n$ such that $x\in P$, $x_0\in P$. We can compute a closed connected dyadic set $\gamma\Subset P$ which contains $x$ and $x_0$. Both $\gamma$ and $P$ are computable, so we can compute an $\ell\geq n$ such that
$\dist(y,\partial P)>2^{1-\ell}d$ for all $y\in\gamma$. Take $Q$ to be $P$ considered as a dyadic polygon of rank $\ell$. Then $Q$ satisfies conditions 1.-3. of Lemma \ref{lem:harnack1}. Thus we can use the algorithm provided by Lemma \ref{lem:harnack1} with input $Q$ to compute the harmonic measure $\omega^\Omega_x$.
\end{proof}

\subsection{Proof of Theorem \ref{thm:lower_computable}}

\begin{lem}\label{lem:interior_to_harmonic}
Let $\Omega\subset [0,\,1]^d$ be a computably regular domain and let  $x\in\Omega$. Assume that $(\Omega, x)$ has a computable interior approximation $\{P_k\}$. Then for some computable subsequence $k_n$, the sequence $\left\{P_{k_n}\right\}$ is a harmonic approximation of $(\Omega, x, \emptyset)$.
\end{lem}
\begin{proof}
Using the computable regularity of $\Omega$, we can compute a sequence $k_n>n+d$ so that
\begin{equation}\label{eq:unif_harm}
\dist(y,\partial\Omega)\leq d2^{-k_n+1}\Rightarrow \omega_y^\Omega(B(x,2^{-n-2}))>1-2^{-n-2}.
\end{equation}

We will show that $\{P_{k_n}\}$ is a harmonic approximation of $(\Omega, x, \emptyset)$.

Indeed, by the definition of the interior approximation, $x\in P_{k_n}$ and $P_{k_n}\cap\partial\Omega=\emptyset$. So the conditions 1) and 3) of the definition of harmonic approximation hold.

To prove the second condition, let us consider any $C^2$-smooth subharmonic function bounded by $1$. Observe that for any $y\in\partial P_{k_n}$, we can use the bound \eqref{eq:unif_harm}  and the fact that $f^{P_{k_n}}(y)=f(y)$ to get:
\begin{multline*}f^{\Omega}(y)-f^{P_{k_n}}(y)\leq\left(\int_{\partial\Omega\cap B(y, 2^{-n-2})}+\int_{\partial\Omega\setminus B(y, 2^{-n-2})}\right)\abs{f(z)-f(y)}\,d\omega_y^{\Omega}(z)\leq\\ 2^{-n-2}+2\times2^{-n-2}<2^{-n}.
\end{multline*}
By the maximum principle applied to harmonic in $P_{k_n}$ function $f^{\Omega}-f^{P_{k_n}}$, the same inequality also holds for $x$.
\end{proof}

\begin{proof}[Proof of the implication 1.$\implies$3. in Theorem \ref{thm:lower_computable}]
Let $\Omega$ be a lower computable computably regular domain with a lower computable boundary. Let $x_0\in\Omega$ be a computable point. By Corollary \ref{cor:compute_inner}, $(\Omega, x_0)$ has a computable interior approximation. By Lemma \ref{lem:interior_to_harmonic}, $(\Omega, x_0,\emptyset)$ also has a harmonic approximation. By Theorem \ref{lem:harm_measure_iff_approx}, the harmonic measure $\omega_{x_0}^\Omega$ is computable. Finally, since $\Omega$ is lower computable, Theorem \ref{thm:uni_comp} implies that the harmonic measure in $\Omega$ is uniformly computable.
\end{proof}

\begin{proof}[Proof of the implication 2.$\implies$1. in Theorem \ref{thm:lower_computable}]
Assume that for $x_0\in\Omega$ the harmonic measure $\omega_{x_0}^\Omega$ is computable. Lemma \ref{lem:lower_computable_boundary} implies that the boundary of $\Omega$ is lower computable.

For a dyadic cube $Q$, let $x_Q$ be its center. Let $\varphi$ be a computable function such that
$$\varphi(x)=
\begin{cases}
1,&\|x\|\leq1/2\\
0,&\|x\|\geq1\\
\geq0\text{ and }\leq1,1\geq\|x\|\leq1/2
\end{cases}
$$
and
\begin{equation}\label{eq:step_function}\varphi_{Q,n}(x):=\varphi(2^{n+1}d(x-x_Q))\end{equation}
Then for an enumeration of all dyadic squares, the sequence $\{\varphi_{Q,n}\}$ is uniformly computable.

Take $k_n$ defined in equation \eqref{eq:unif_harm} and let $C=C(k_n,d)$ be the constant from Lemma \ref{uniform_harnack_bound}.
By uniform computability of $\{\varphi_{Q,n}\}$ there is a uniformly computable sequence $M_{{Q,n}}$ such that
\begin{equation}\label{eq:compute_interior}
\abs{M_{{Q,n}}-\binr{\omega_{x_0}^\Omega}{\varphi_{Q,n}}}<C^{-1}2^{-n-2}
\end{equation}

Call a cube $Q$ of rank $k_n$ \emph{interior} if $M_{{Q,\ n}}<C^{-1}2^{-n-1}$. Thus $\binr{\omega_{x_0}^\Omega}{\varphi_{Q,n}}<C^{-1}2^{-n}$.

Observe that if $x$ is an interior point, then for some $n$, $\dist(x,\partial\Omega)>3d2^{-n}$. The dyadic cube $Q$ of rank $k_n$ containing $x$ is interior, since $\binr{\omega_{x_0}^\Omega}{\phi_{Q,n}}=0$ and $\abs{M_{{Q,n}}}<C^{-1}2^{-n-2}$. Thus every interior point of $\Omega$ is contained inside an interior cube.

Let $P_n$ be the maximal dyadic polygon consisting of interior cubes and containing $x_0$. $P_n$ can be computed uniformly in $n$ by computing $M_{{Q,n}}$ for all dyadic cubes of rank $k_n$.  By the previous observation, $\Omega\subset\cup_n P_n$.

Let us show that $P_n\subset \Omega$ and, moreover
\begin{equation}\label{eq:far}
\dist(x,\partial\Omega)>d2^{-k_n+1},\quad x\in P_n
\end{equation}
Indeed, let $z$ be the closest to $x_0$ point of $P_n$ for which $\dist(z,\partial\Omega)=d2^{-k_n+1}$. Let $Q$ be the dyadic cube of rank $k_n$ containing $x$. Then, by Lemma  \ref{uniform_harnack_bound}, since for $y\in Q$, $\dist(y,\partial\Omega)\geq d2^{-k_n}$, we have

$$
\omega_{z}^\Omega(B(z, 2^{-n}))\leq\binr{\omega_{x_Q}^\Omega}{\phi_{Q,\ n}}\leq C\binr{\omega_{x_0}^\Omega}{\phi_{Q,\ n}}<2^{-n}
$$
which, by the definition of $k_n$ in \eqref{eq:unif_harm}, means that $dist(z,\partial\Omega)> d2^{-k_n+1}$. This contradiction proves that $P_n\subset\Omega$ and so $\Omega=\cup_n P_n$. Since $\{P_n\}$ is uniformly computable, $\Omega$ is lower computable.
\end{proof}

\subsection{Proof of Theorem \ref{thm:computable_boundary}}
\begin{proof}[Proof of Theorem \ref{thm:computable_boundary}]
Since a domain with computable boundary is lower computable, the implication 1.$\implies$ 3. follows from  Theorem \ref{thm:lower_computable} and the implication 2.$\implies$3. follows from Theorem \ref{thm:uni_comp}. The implication 3.$\implies$2. is trivial.

To prove the remaining implication 3.$\implies$1., first note that by compactness and regularity of $\Omega$, there exists a sequence $k_n$ satisfying \eqref{eq:unif_harm}. Let us describe an algorithm for uniformly computing such $k_n$, which would imply the computable regularity of $\Omega$. Define ${\mathcal P}_n$ to be the collection of dyadic cubes of rank $n+d$ such that one of their dyadic neighbours intersect $\partial\Omega$. By computability of $\partial\Omega$, ${\mathcal P}_n$ is uniformly computable.

Given $n$, consider the finite collection $\{\varphi_{Q,n}\}_{Q\in{\mathcal P}_n}$, where $\varphi_{Q,n}$ are defined in \eqref{eq:step_function}. Notice that the sequence $\{\varphi_{Q,n+1}\}$ is uniformly computable as a sequence indexed by $Q$ and $n$.

By the uniform computability of harmonic measure, the harmonic functions
$$u_{Q,n}(x):=\binr{\omega_x^\Omega}{\phi_{Q,n+3}}$$
are uniformly computable. Thus one can uniformly compute $\varepsilon(n)$ such that
\begin{equation}\label{eq:mod_continuity}
x,y\in\Omega;\ \|x-y\|<\varepsilon(n)\implies\abs{u_{Q,n}(x)-u_{Q,n}(y)}<2^{-n-3}.
\end{equation}

Start with $k$ such that $d2^{-k}<\varepsilon(n)$ and $k>n+d$. Note that for any dyadic cube $Q'\in{\mathcal P}_k$ with rank $k$ there is $Q\in{\mathcal P}_n$ with $Q'\subset Q$. Check that for all such dyadic cubes $Q'$ of rank $k$ the following property holds:
\begin{equation}\label{eq:stopping_time}
Q'\in{\mathcal P}_k,\ Q\in{\mathcal P}_n,\ Q'\subset Q\implies\binr{\omega_{x_Q'}^\Omega}{\phi_{Q,n+3}}=u_{Q,n}(x_Q)>1-2^{-n-3}
\end{equation}
If the condition \eqref{eq:stopping_time} fails for one of the $Q'$, increase $k$ by $1$ and repeat. By regularity of $\Omega$ and compactness, the process will eventually stop.

When the process stop, we compute $k$ such that the condition \eqref{eq:stopping_time} holds. Take $k_n=k+d$. If $\dist(x,\partial\Omega)\leq d2^{-k_n+1}$ then the dyadic cube $Q'$ of rank $k$ with $x\in Q'$ must satisfy $Q'\in{\mathcal P}_k$. Thus, we have, using \eqref{eq:mod_continuity} and \eqref{eq:stopping_time}
\begin{multline*}
\omega_x^\Omega(B(x, 2^{-n-2}))\geq u_{Q,n}(x)\geq u_{Q,n}(x_Q)-\abs{u_{Q,n}(x_Q)-u_{Q,n}(x)}>\\1-2\times2^{-n-3}=1-2^{-n-2}
\end{multline*}
which is exactly \eqref{eq:unif_harm}.
\end{proof}

\subsection{Proofs of Theorems \ref{thm:nocompute_measure} and \ref{thm:nocompute_Dirichlet}}\label{sec:negative}

Now we are ready to describe the construction which will serve as the basis for all the proofs in this section.

Let $x_n:=1-2^{-n}$. Fix a non-computable lower-computable set $B$ with a computable enumeration $\left\{b_k\right\}_{k=1}^\infty$.

Let $r_n^k:=\exp\left(-8^k\right)$ for $k\geq n$. Observe that by Lemma \ref{observation} applied to $K=B(x_n, r_n^k)$ and the domain $\DD\setminus K\subset \DD$, we have the estimate

\begin{equation}\label{eq:up1}\omega^{\DD\cap B(x_n, r_n^k)}_x(x_n+r_n^kS^1)<4^{-k}, \text{ if } x\in\DD\text{ and }  |x-x_n|>4^{-k}.
\end{equation}

Let us define
\begin{equation}
\Omega_0:=\DD\setminus\cup_{k\geq1}\left\{y\,:\,|y-x_k|\leq r_k^k\right\}
\end{equation}

The domains we construct in this section will be contained in the unit disk and will contain the domain $\Omega_0$.
More specifically, we will consider the domains of the form
\begin{equation}\label{eq:def:domain}
\Omega^*:=\DD\setminus\cup_{n} C_n
\end{equation}
with $C_n$ defined by
\begin{equation}\label{eq:def:domain:more}
C_n=\begin{cases}
A^{k+n}_n,& n=b_k\\
A^\infty_n,&n\notin B
\end{cases}
\end{equation}
where $A^k_n\subset B(x_n, r_n^n)$,  $A^\infty_n\subset B(x_n, r_n^n)$ be uniformly computable compacts, which satisfy the property
\begin{equation}\label{eq:close_to_infinity}
\omega_0^{\DD\setminus (A^k_n\cup A^\infty_n)}\left(A^k_n\Delta A^\infty_n\right)\leq 2^{-n-k}
\end{equation}
\begin{figure}
\centering
  \includegraphics[width=.3\linewidth]{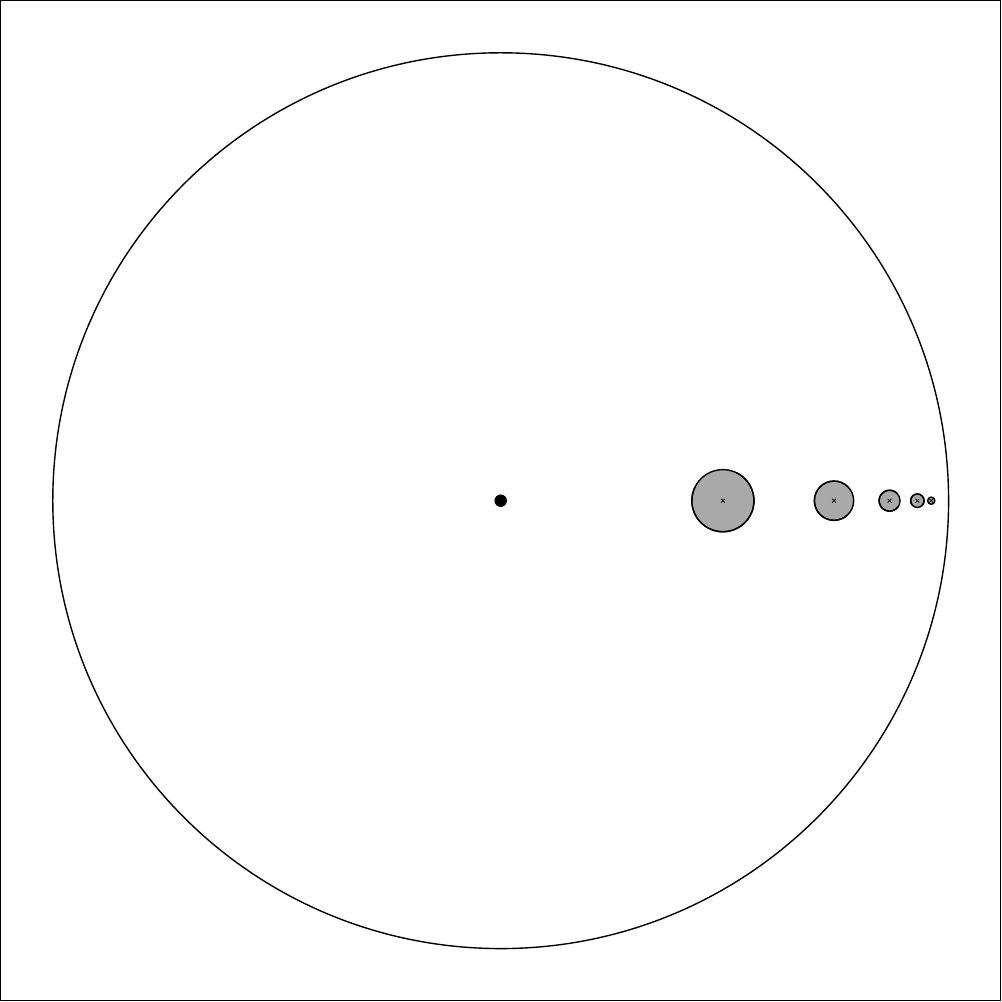}
  \caption{\label{fig:Omega_0} The domain $\Omega_0$}
\end{figure}
Let us observe that since the sequences $\left\{A^k_{b_k}\right\}$ and $\left\{A^\infty_n\right\}$ are uniformly computable, $\partial\Omega^*$ will always be lower computable. Actually, even more is true, as we show in the next lemma.

\begin{lem}\label{lem:harm_measure_computable}
For any domain $\Omega^*$ of the form \eqref{eq:def:domain} and any $x\in\Omega^*$, the harmonic measure $\omega_x^{\Omega^*}$ is computable.
\end{lem}
\begin{proof}
By Theorem \ref{thm:allcompute}, it is enough to prove that $\omega_0^{\Omega^*}$ is computable.

Let $D_n$ be the unit disk with a thin semi-strip removed
$$D_n:=\DD\setminus\{z\,:\,\Re z\geq1-2^{-n}-r^n_n;\ |\Im z|\leq r^n_n\}$$
Note that
\begin{equation}\label{eq:contain}
\Omega^*\setminus \cup_{k=1}^n C_k\subset D_n
\end{equation}

We can explicitly compute sequence $\{m_n\}$ such that
$$\omega_0^{D_{m_n}}(\DD\setminus D_{m_n})<2^{-n-3}.$$

Let us define the domain
$$\Omega'_n:=D_{m_n}\setminus\cup_{\ell=1}^{m_n }C'_\ell$$
where
$$C'_{\ell}:=\begin{cases}
A^k_n,& n=b_k; k\leq n+3\\
A^\infty_n,&\text{ otherwise}
\end{cases}$$
$\Omega'_n$ is a finitely-connected planar domain. $\partial\Omega'_n$ is computable (since to compute it one only uses the first $n+3$ elements of $\{b_k\}$).
Utilizing Theorem \ref{lem:harm_measure_iff_approx} and Lemma \ref{lem:interior_to_harmonic}, for each $n$ one can compute interior harmonic approximation $\left\{P_k^n(\Omega'_n)\right\}_{k=1}^\infty$. The sequence $\left\{R_n\right\}_{n=1}^\infty\left\{P_{n+1}^n(\Omega'_n)\right\}_{n=1}^\infty$ is therefore uniformly computable.

Observe that by the maximum principle and relations \eqref{eq:close_to_infinity}, \eqref{eq:contain}, we have
$$\omega_0^{\Omega^{*}}(\partial\Omega^{*}\cap\Omega'_n)\leq 2^{-n-3}+\sum_{k=n+3}^\infty 2^-k=2^{-n-2}$$

Thus, by Lemma \ref{lem:compare_harmonic}, for any continuous $f$ bounded by $1$, we have
$$\abs{\binr{\omega_{x}^{\Omega^{*}}}{f}-\binr{\omega^{\Omega'_n}_{x}}{f}}\leq2\omega_x^{\Omega^{*}}(\partial\Omega^{*}\cap\Omega'_n)\leq 2^{-n-1}.$$

By the definition of harmonic approximation, this implies that for any subharmonic smooth 1-Lipshitz function bounded by $1$ we have
$$\abs{\binr{\omega_{x}^{\Omega^{*}}}{f}-\binr{\omega^{R_n}_{x}}{f}}\leq\abs{\binr{\omega_{x}^{\Omega^{*}}}{f}-\binr{\omega^{\Omega'_n}_{x}}{f}}+\abs{\binr{\omega_{x}^{R_n}}{f}-\binr{\omega^{\Omega'_n}_{x}}{f}}< 2^{-n}$$

Also, by the construction of $\Omega^{*}$, we have
$$\omega_0^{\Omega^{*}}(\partial\Omega^{*}\cap \overline{R_n})\leq\sum_{k=n+3}^\infty \omega_0^{\DD\setminus (A^k_n\cup A^\infty_n)}\left(A^k_n\Delta A^\infty_n\right)< 2^{-n-3}$$

So $\left\{R_n\right\}$ is a harmonic approximation of $(\Omega,0,\emptyset)$, so, by Theorem \ref{lem:harm_measure_iff_approx}, the measure $\omega_0^{\Omega^*}$ is computable.
\end{proof}

\begin{proof}[Proof of Theorems \ref{thm:nocompute_measure} and \ref{thm:nocompute_Dirichlet}]

Let $f$ be any nonnegative, bounded by $1$ computable function such that $f$ is equal to $1$ on the arc $\left\{z\,:\,|z|=1, \Re z\leq -1/2\right\}$ and equal to $0$ for all $z$ with $\Re z\geq0$.

We will use domains
$$E_n:=\DD\setminus\left\{z\,:\,\Re z\geq0;\ |z|\leq1-2^{-n}-2^{-n-1}\text{ or }|z|\geq1-2^{-n-1}-2^{-n-2}\right\}$$
(see Figure \ref{fig:En}).
\begin{figure}
\centering
\begin{subfigure}{.5\textwidth}
\centering
{\includegraphics[width=0.4\linewidth]{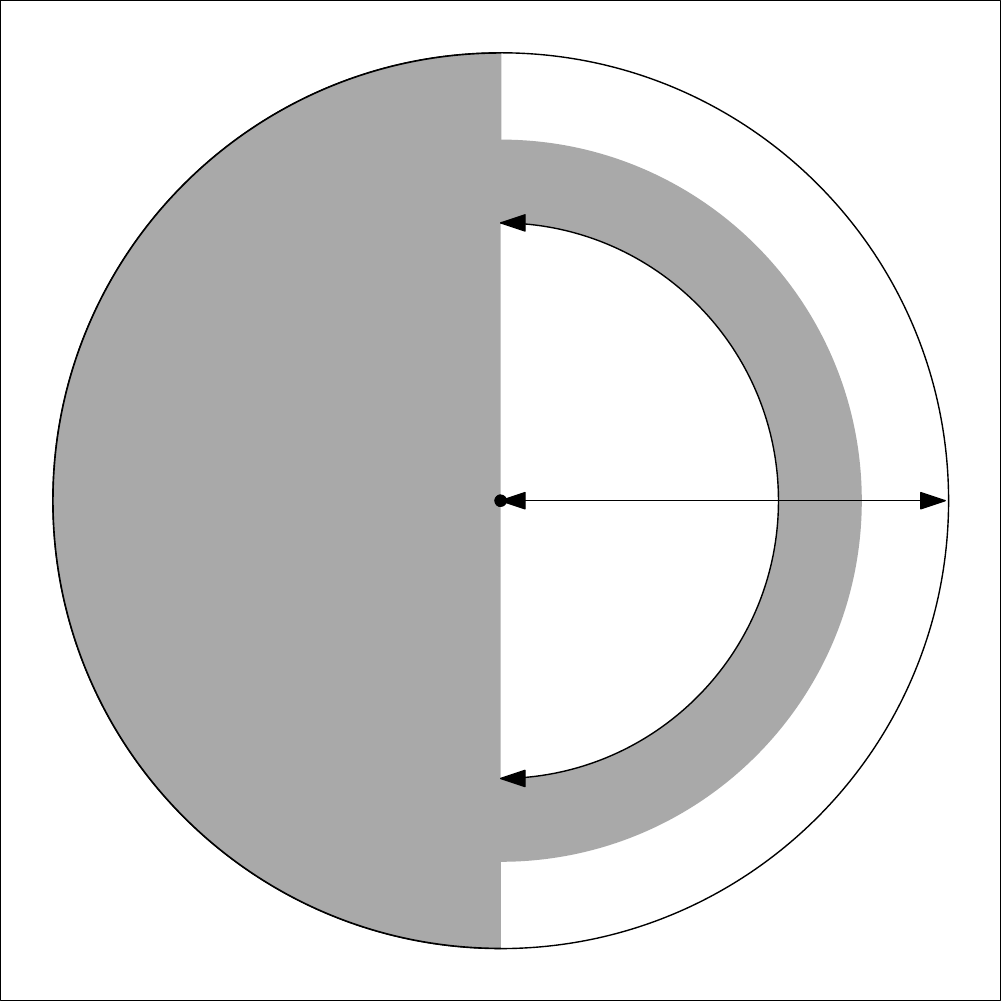}}
\caption{\label{fig:En} The set $E_n$.
}
\end{subfigure}%
\begin{subfigure}{.5\textwidth}
\centering
  \includegraphics[width=.4\linewidth]{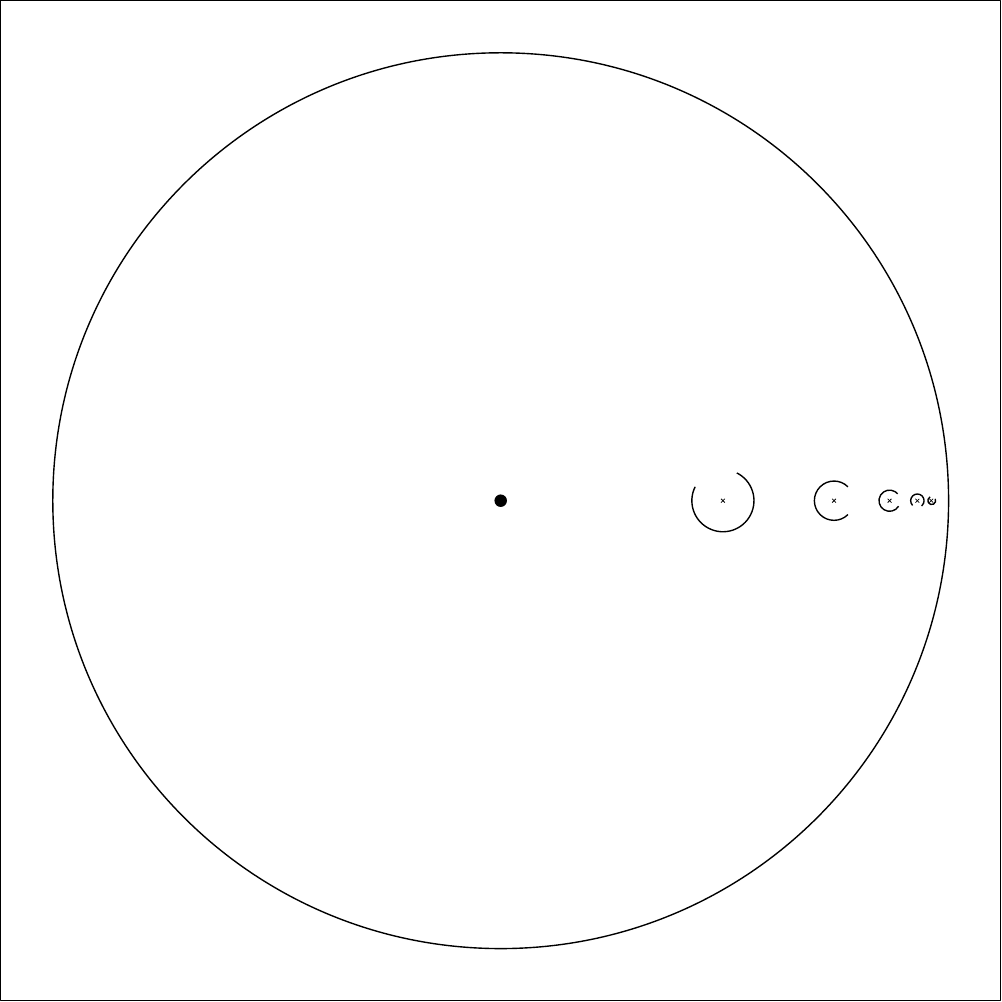}
  \caption{\label{fig:Omega^*} The domain $\Omega^*$ used in the proof of Theorems \ref{thm:nocompute_measure} and \ref{thm:nocompute_Dirichlet}}
\end{subfigure}
\centering
\caption{The domains used in the proof of Theorems \ref{thm:nocompute_measure} and \ref{thm:nocompute_Dirichlet}}
\end{figure}

We can compute a sequence $\ell_n$ such that
\begin{equation}\label{eq:lower}\binr{\omega_{x_n}^{E_n}}f=\omega_{x_n}^{E_n}\left\{z\,:\,|z|=1, \Re z\leq -1/2\right\}>2^{-\ell_n}\end{equation}

In our construction \eqref{eq:def:domain}, we take $A^\infty_n=\emptyset$. We can uniformly compute an arc $A_n^k$  of the circle
$x_n+r^k_n S^1$ which satisfy the condition
\begin{equation}\label{eq:TheoremC}
\omega_{x_n}^{B(x_n, r^k_n)}(A_n^k)>1-2^{-\ell_n-2}
\end{equation}

This sequence $A_n^k$ satisfy the  conditions of \eqref{eq:def:domain}, since, by \eqref{eq:up1},
$$\omega_0^{\DD\setminus (A^k_n\cup A^\infty_n)}\left(A^k_n\Delta A^\infty_n\right)=\omega_0^{\DD\setminus (A^k_n)}\left(B(x_n, r^k_n)\right)< 4^{-k}\leq 2^{-n-k}$$

Thus, by Lemma \ref{lem:harm_measure_computable}, $\binr{\omega_x^{\Omega^{*}}}f$ is computable for any $x\in\Omega$.

Note now that if $n\notin B$, then $E_n\subset\Omega$, and, by maximum principle and \eqref{eq:lower}
\begin{equation}\label{eq:bellow}
\binr{\omega_{x_n}^{\Omega^{*}}}f\geq\binr{\omega_{x_n}^{D_n}}f\geq2^{-\ell_n}
\end{equation}

On the other hand, if $n\in B$, then, since for some $k$, $\Omega^*\subset\DD\setminus A^k_n$, we can use \eqref{eq:TheoremC} and the maximum principle to obtain
\begin{equation}\label{eq:above}
\binr{\omega_{x_n}^{\Omega^{*}}}f<2^{-\ell_n-2}
\end{equation}

Assume that $\binr{\omega_{x_n}^{\Omega^{*}}}f$ is uniformly computable by some algorithm. Then we can use the algorithm to uniformly compute $2^{-\ell_n-1}$ approximation of $\binr{\omega_{x_n}^{\Omega^{*}}}f$. By \eqref{eq:bellow} and \eqref{eq:above}, it would allow us to decide whether $n\in B$. This contradicts no-computability of $B$ and proves Theorem C.

Now we just need to observe that Theorem \ref{thm:nocompute_Dirichlet} implies Theorem \ref{thm:nocompute_measure}.
\end{proof}

\begin{lem}\label{lem:non_comp_regular}
There exists a lower computable regular domain $\Omega$ with lower computable boundary and uniformly computable harmonic measure which is not computably regular.
\end{lem}
\begin{proof}
Here we use the construction \eqref{eq:def:domain} with $A^\infty_n:=B(x_n, r^n_n)$ and for $n\leq k<\infty$
$$A^k_{n}=\left\{z\,:\,r^n_n\geq|z-x_n|\geq r^n_n/2;|\Im z|\geq d^k_n \right\}\cup B(x_n+r^k_n/10, e^k_n),$$
(see Figure \ref{fig:Ank}), where $d^k_n$ and $e^k_n$ are uniformly computed such that
\begin{eqnarray}
\omega_0^{\DD\setminus B(x_n, r_n^n)}(\left\{z\,:\,|z-x_n|=r^n_n;|\Im z|\leq d^k_n \right\})&<&2^{-k-n}\label{eq:bounds1}\\
\omega_{x_n}^{\DD}(B(x_n+r^k_n/10, e^k_n))&<&1/2\label{eq:bounds3}
\end{eqnarray}
\begin{figure}
\begin{center}
{\includegraphics[width=0.3\linewidth]{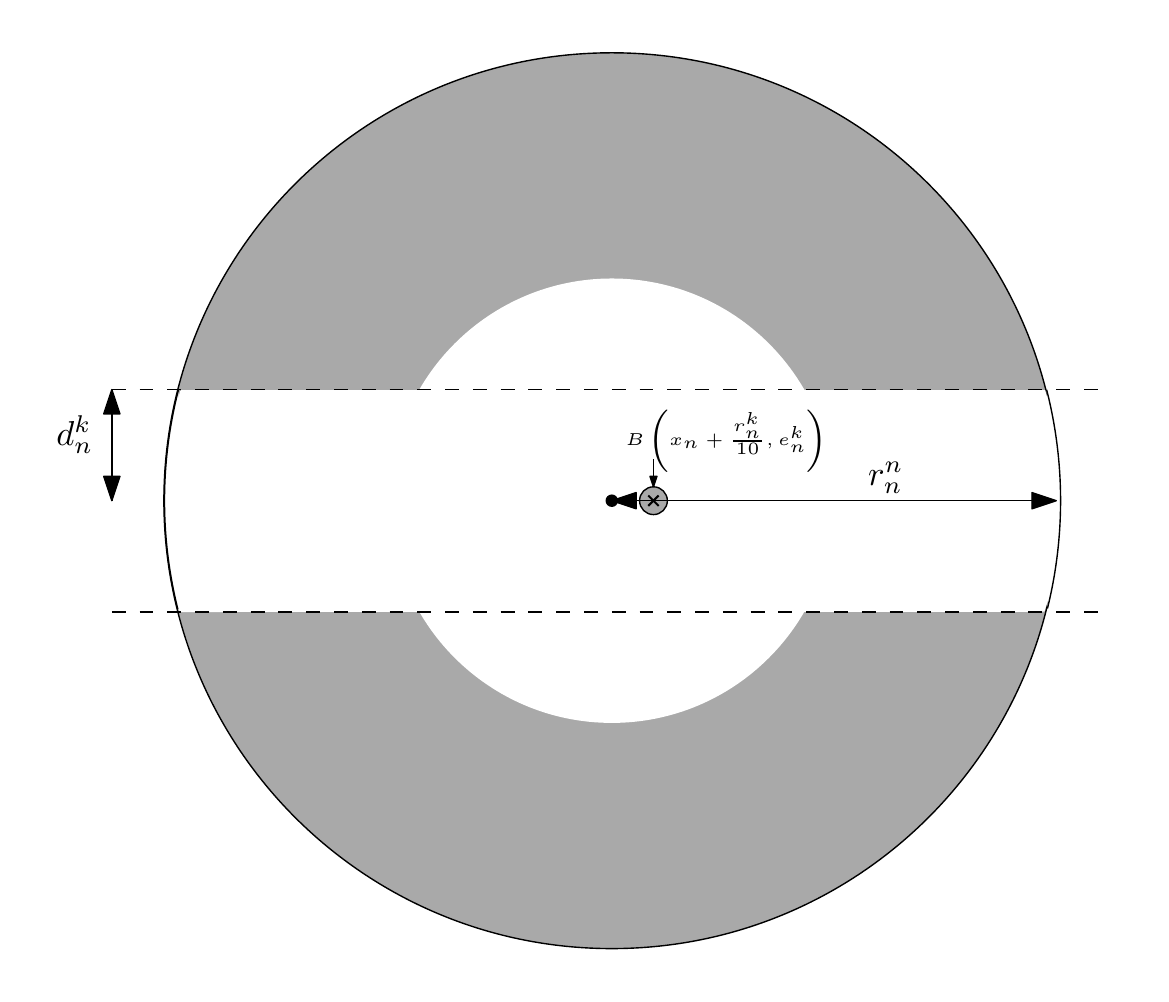}}
\caption{\label{fig:Ank} The set $A_n^k$ used in the proof of Lemma \ref{lem:non_comp_regular}.
}
\end{center}
\end{figure}
\eqref{eq:bounds1} implies that the condition \eqref{eq:close_to_infinity} is satisfied.

Also
$$\Omega^*=\Omega_0\cup\bigcup_k B(x_{b_k}, r^{b_k}_{b_k})\setminus A^{k+b_k}_{b_k}$$
so the domain $\Omega^*$ is lower computable. By Theorem \ref{thm:uni_comp} and Lemma \ref{lem:harm_measure_computable}, the harmonic measure of $\Omega^*$ is uniformly computable.

Note that by the maximum principle and \eqref{eq:bounds3},
\begin{equation}\label{eq:lowlow}
\omega_{x_n}^{\Omega^*}(B(x_n+r^k_n/10, e^k_n))<1/2\end{equation}
Assume that $\Omega^*$ is computably regular. Then we can compute $\eps(n)$ so that
$$
\dist(x,\partial\Omega)<\eps(n)\Rightarrow \omega_x^\Omega(B(x,r^n_n/4))>1-2^{-n}.
$$
and that would allow us to uniformly compute $k(n)$ such that
$$r^{k(n)+n}_n<\eps(n)$$

By \eqref{eq:lowlow}, if $n\in B,\ n=b_k$ then $\omega_{x_n}^\Omega(B(x_n,r^n_n/4))<1/2$ and $\dist(x_n,\partial\Omega)<r^{k+b_k}_n/4$, so $k\leq k(n)$. Thus to decide whether $n\in B$, we only need to know first $k(n)$ members of the sequence $\left\{b_k\right\}$. This contradicts non-computability of $B$.

\end{proof}


\begin{thebibliography}{10}

\bibitem{BB07}
I.~Binder and M.~Braverman.
\newblock Derandomization of euclidean random walks.
\newblock In {\em APPROX-RANDOM}, pages 353--365, 2007.

\bibitem{BBRY}
I.~Binder, M.~Braverman, C.~Rojas, and M.~Yampolsky.
\newblock Computability of {B}rolin-{L}yubich measure.
\newblock {\em Comm. Math. Phys.}, 308(3):743--771, 2011.

\bibitem{Binder2019}
I.~Binder, C.~Rojas, and M.~Yampolsky.
\newblock Carath{\'e}odory convergence and harmonic measure.
\newblock {\em Potential Anal.}, 51:499--509, 2019.

\bibitem{BRY-car}
Ilia Binder, Cristobal Rojas, and Michael Yampolsky.
\newblock Computable {C}arath\'eodory theory.
\newblock {\em Adv. Math.}, 265:280--312, 2014.

\bibitem{BB}
E.~Bishop and D.~S. Bridges.
\newblock {\em Constructive {A}nalysis}.
\newblock Springer-Verlag, Berlin, 1985.

\bibitem{BY-book}
M~Braverman and M.~Yampolsky.
\newblock {\em Computability of {J}ulia sets}, volume~23 of {\em Algorithms and
  {C}omputation in {M}athematics}.
\newblock Springer, 2008.

\bibitem{Car12}
C.~Carath\'eodory.
\newblock Untersuchungen \"uber die konformen {A}bbildungen von festen und
  ver\"anderlichen {G}ebieten.
\newblock {\em Math. Ann.}, 72(1):107--144, 1912.

\bibitem{Subharmonic}
W.~K. Hayman and P.~B. Kennedy.
\newblock {\em Subharmonic functions. {V}ol. {I}}.
\newblock Academic Press [Harcourt Brace Jovanovich, Publishers], London-New
  York, 1976.
\newblock London Mathematical Society Monographs, No. 9.

\bibitem{Koebe}
P.~Koebe.
\newblock {\"U}ber eine neue {M}ethode der konformen {A}bbildung und
  {U}niformisierung.
\newblock {\em Nachr. {K}{\"o}nigl. {G}es. {W}iss. {G}{\"o}ttingen, {M}ath.
  {P}hys. {K}l.}, pages 844--848, 1912.

\bibitem{Tur}
A.~M. Turing.
\newblock On computable numbers, with an application to the
  {E}ntscheidungsproblem.
\newblock {\em Proceedings, London Mathematical Society}, pages 230--265, 1936.

\bibitem{Zhong1998}
Ning Zhong.
\newblock Derivatives of computable functions.
\newblock {\em Mathematical Logic Quarterly}, 44(3):304--316, 1998.

\end{thebibliography}
\end{document}